\documentclass[oneside, 11pt]{article}
\usepackage{setspace}
\usepackage{titlesec}
\usepackage[utf8]{inputenc}
\usepackage{tabularx}	
\usepackage{etoolbox}
\usepackage{mathtools}
\usepackage{amsmath}
\usepackage{amssymb} 
\usepackage{amsfonts}
\usepackage{mathrsfs} 
\usepackage{amsthm,pifont}
\usepackage{graphicx}
\usepackage{bbm}
\usepackage[colorlinks=true, pdfstartview=FitV, linkcolor=blue,
citecolor=blue, urlcolor=blue, pagebackref=false]{hyperref}
\usepackage{cleveref}
\usepackage[nottoc,numbib]{tocbibind}
\usepackage{xcolor}

\usepackage{tikz}

\usepackage{cite}
\usepackage{braket}
\usepackage{leftidx,tensor}
\usepackage{dsfont}
\usepackage{calc}
\usepackage{caption}
\usepackage{subcaption}
\usepackage{color}
\usepackage[printonlyused,nohyperlinks]{acronym}
\usetikzlibrary{decorations.markings}
\usetikzlibrary{shapes.geometric}
\usetikzlibrary{patterns}
\tikzstyle{vertex}=[circle, draw, inner sep=0pt, minimum size=6pt]
\newcommand{\vertex}{\node[vertex]}

\usepackage[utf8]{inputenc}
\usepackage[english]{babel}
\usepackage{enumitem}

\allowdisplaybreaks

\newtheorem{theorem}{Theorem}[section]
\makeatletter
\patchcmd{\ttlh@hang}{\parindent\z@}{\parindent\z@\leavevmode}{}{}
\patchcmd{\ttlh@hang}{\noindent}{}{}{}
\makeatother
\titleformat*{\section}{\large\bfseries}
\titleformat*{\subsection}{\small\bfseries}
\titleformat*{\subsubsection}{\small\bfseries}
\titleformat*{\paragraph}{\small\bfseries}
\titleformat*{\subparagraph}{\small\bfseries}

\newcommand{\N}{\mathbb{N}}

\newcommand{\Z}{\mathbb{Z}}
\newcommand{\E}{\mathbb{E}}
\newcommand{\Ec}{\mathbb{E}_{\beta_c}}
\newcommand{\p}{\mathbb{P}}
\newcommand{\pc}{\mathbb{P}_{\beta_c}}
\newcommand{\Km}{K_{\text{max}}}

\newcommand{\eps}{\varepsilon}

\newtheorem{lemma}[theorem]{Lemma}

\newtheorem{proposition}[theorem]{Proposition}

\usepackage{geometry}
\begin{document}

	\title{Isoperimetric lower bounds for critical exponents for long-range percolation}

	\author{Johannes B\"aumler\footnote{ \textsc{Department of Mathematics, TU Munich, Germany}. E-Mail: \href{mailto:johannes.baeumler@tum.de}{johannes.baeumler@tum.de}} \ and Noam Berger\footnote{ \textsc{Department of Mathematics, TU Munich, Germany}. E-Mail: \href{mailto:noam.berger@tum.de}{noam.berger@tum.de} 
	}
	}
	
	\maketitle
	
	\begin{center}
		\parbox{13cm}{ \textbf{Abstract.} We study independent long-range percolation on $\mathbb{Z}^d$ where the vertices $x$ and $y$ are connected with probability $1-e^{-\beta\|x-y\|^{-d-\alpha}}$ for $\alpha > 0$. Provided the critical exponents $\delta$ and $2-\eta$ defined by $\delta = \lim_{n\to \infty} \frac{-\log(n)}{\log\left(\mathbb{P}_{\beta_c}\left(|K_0|\geq n\right)\right)}$  
			and $2-\eta = \lim_{x \to \infty} \frac{\log\left(\mathbb{P}_{\beta_c}\left(0\leftrightarrow x\right)\right)}{\log(\|x\|)} + d$ exist, where $K_0$ is the cluster containing the origin, we show that
			\begin{equation*}
			\delta \geq \frac{d+(\alpha\wedge 1)}{d-(\alpha\wedge 1)} \ \text{ and } \ 2-\eta \geq \alpha \wedge 1 \text.
			\end{equation*}
			The lower bound on $\delta$ is believed to be sharp for $d = 1, \alpha \in \left[\frac{1}{3},1\right)$ and for $d = 2, \alpha \in \left[\frac{2}{3},1\right]$, whereas the lower bound on $2-\eta$ is sharp for $d=1, \alpha \in (0,1)$, and for $\alpha \in \left(0,1\right]$ for $d>1$, and is not believed to be sharp otherwise.
			Our main tool is a connection between the critical exponents and the isoperimetry of cubes inside $\mathbb{Z}^d$.
	}
	\end{center}
	
	\vspace{0.1cm}
	
	\hypersetup{linkcolor=black}
	\tableofcontents
	\hypersetup{linkcolor=blue}

\section{Introduction}
\let\thefootnote\relax\footnotetext{{\sl MSC Class}: 60K35, 82B27, 82B43}
\let\thefootnote\relax\footnotetext{{\sl Keywords}: Long-range percolation, phase transition, critical exponents}

Consider Bernoulli bond percolation on $\Z^d$ where we include an edge between the vertices $x,y\in \Z^d$ with probability $1-e^{-\beta J(x,y)}$ and independent of all other edges. The function $J:\Z^d\times \Z^d \rightarrow \left[0,\infty\right)$ is a kernel that is symmetric, i.e., $J(x,y)=J(y,x)$ for all $x,y\in \Z^d$. We denote the resulting probability measure by $\p_\beta$ and its expectation by $\E_\beta$. Edges that are included are also referred to as open. We are interested in the case where the kernel is also translation invariant and integrable, meaning that $J(x,y)=J(0,y-x)$ for all $x,y\in \Z^d$ and $\sum_{x\in \Z^d} J(0,x)< \infty$. The integrability condition guarantees that the resulting graph is almost surely locally finite. This procedure creates certain clusters, which are the connected components in the resulting random graph. Write $K_x$ for the cluster containing the vertex $x\in \Z^d$. A major question in percolation theory is the emergence of infinite clusters, for which we define the critical parameter $\beta_c$ by
\begin{equation*}
\beta_{c} = \inf \left\{\beta \geq 0 : \p_\beta \left( |K_0| = \infty \right) > 0 \right\} \text.
\end{equation*}
A comparison with a Galton-Watson tree shows that there are no infinite clusters for $\beta < \left(\sum_{x\in \Z^d} J(0,x)\right)^{-1}$, which shows $\beta_{c}>0$. For $d > 1$ and $J \neq 0$ it is well-known that $\beta_c < \infty$, whereas for $d=1$ it is known that $\beta_c < \infty$ in the case where $J(x,y) \simeq \|x-y\|^{-1-\alpha}$ for $\alpha \leq 1$ \cite{newman1986one,duminil2020long}, whereas $\beta_c=\infty$ for $\alpha>1$. Long-range percolation mostly deals with the case where $J(x,y) \simeq \|x-y\|^{-d-\alpha}$ for some $\alpha >0$, where we write $J(x,y)\simeq \|x-y\|^{-d-\alpha}$ if the ratio between them satisfies $\eps < \frac{J(x,y)}{\|x-y\|^{-d-\alpha}} < \eps^{-1}$ for a small enough $\eps>0$ and $\|x-y\|$ large enough.  In general it is expected that for $\alpha > d$ the resulting graph looks similar to nearest-neighbor percolation, is very well connected for $\alpha <d$, and shows a self-similar behavior for $\alpha = d$. See \cite{baeumler2022behavior,baeumler2022distances,benjamini2001diameter,berger2004lower,biskup2004scaling,biskup2011graph, ding2013distances} for results pointing in this direction. 
From the definition of $\beta_c$ and the standard Harris coupling \cite{heydenreich2017progress} we see that $ \p_\beta \left( |K_0| = \infty \right) > 0$ for $\beta > \beta_c$ and $ \p_\beta \left( |K_0| = \infty \right) = 0$ for $\beta < \beta_c$, but it is not clear what happens at $\beta = \beta_c$. For $J(x,y) \simeq \|x-y\|^{-d-\alpha}$ with $\alpha \in (0,d)$ and all $d \in \N_{>0}$ the second author showed that $ \pc \left( |K_0| = \infty \right) = 0$ \cite[Theorem 1.5]{berger2002transience}, whereas for $d=1$ and $J(x,y)\simeq \|x-y\|^{-2}$ it is a result by Aizenman and Newman that $ \pc \left( |K_0| = \infty \right) > 0$ \cite{aizenman1986discontinuity}. For $d\geq 2$ and $\alpha \geq d$ it is also expected that  $\pc \left( |K_0| = \infty \right) = 0$, but there is no full proof known at the moment. Whenever there is no infinite cluster at the critical value, it is a central question how fast the tail of the cluster at criticality $\pc \left(|K_0|\geq n\right)$ and the two-point function $\pc \left(0 \leftrightarrow x\right)$ tend to $0$ as $n$, respectively $\|x\|$, grow. Here we write $x\leftrightarrow y$ if there exists an open path from $x$ to $y$.
It is conjectured that
\begin{align}\label{eq:clustersize}
&\pc \left(|K_0|\geq n\right) \approx n^{-1/\delta}   \ \ \ \ \ \ \ \ \ \ \ \ \ \ \text{ as } n\to \infty,\\
\label{eq:twopointfct}
&\pc \left(0 \leftrightarrow x\right) \approx \|x\|^{-d+2-\eta}  \ \ \ \ \ \ \ \ \ \ \text{ as } \|x\| \to \infty
\end{align}
for certain numbers $\eta,\delta$ depending on $d$ and $\alpha$, but not on the precise details of the kernel $J$. Here, we write $f(n)\approx n^c$ if $f(n)=n^{c+o(1)}$. Even the existence of the exponents is not clear and it is still open, whether the limits $\lim_{n\to \infty}\frac{\log \left(\pc \left(|K_0|\geq n\right)\right)}{\log(n)}$ and $\lim_{\|x\|\to \infty}\frac{\log \left(\pc \left(0 \leftrightarrow x\right)\right)}{\log(\|x\|)}$ exist. The widely accepted conjecture is that they exist. This has been for example proven for other models of percolation like two-dimensional percolation on the triangular lattice \cite{lawler2002one,smirnov2001,smirnov2001critical} or percolation for high enough dimension $d$, or for small enough $\alpha$ \cite{heydenreich2008mean}. Recently, Hutchcroft proved the upper bounds $\delta \leq \frac{2d}{d-\alpha}$ and $2-\eta \leq \alpha$ \cite{hutchcroft2022sharp}, improving his previous result $\delta \leq \frac{2d+\alpha}{d-\alpha}$ \cite{hutchcroft2021power} which is, to our knowledge, the first rigorous proof of a power-law decay of $\pc \left(|K_0|\geq n\right)$ for long-range percolation.

\paragraph*{Our results} In this paper, we give lower bounds on the exponents $\delta$ and $2-\eta$. We will always assume an upper bound on the kernel $J$  of the form $J(x,y)\leq C_1 \|x-y\|^{-d-\alpha}$ for some constant $C_1< \infty$.

\begin{theorem}\label{theo:clustersize}
	Let $\alpha \in (0,1)$ for $d=1$, respectively $\alpha > 0$ for $d>1$. Suppose that $J(x,y) \leq C_1 \|x-y\|^{-d-\alpha}$
	and the exponent $\delta$ defined in \eqref{eq:clustersize} exists. Then
	\begin{align*}
	\delta \geq \frac{d+ (\alpha \wedge 1)}{d-(\alpha \wedge 1)}\text.
	\end{align*}
\end{theorem}
\Cref{theo:clustersize} is an immediate consequence of \Cref{propo:cluster decay}. It is only of interest in dimension $d \in \{1,2\}$ and for $\alpha > \frac{d}{3}$, as it is known in wider generality that $\delta \geq 2$ \cite[Proposition 10.29]{grimmett1999percolation,aizenman1987sharpness}.
For the case where $d=1$ and $\alpha \in \left[\frac{1}{3},1\right)$, respectively where $d=2$ and $\alpha \in \left[\frac{2}{3},1\right]$, our lower bound coincides with the conjectured true value of $\delta$. \\
In particular, \Cref{theo:clustersize} shows that for $d \in \{1,2\}$ and $\alpha > \frac d3$ the model does not exhibit the so called 'mean-field behavior'. The notion of 'mean-field  behavior' is a notion that comes from physics, and roughly means that all the critical exponents are the same as in models of infinite dimension, such as Erd\"os-R\'enyi graphs (in the $n\to\infty$ limit) or the binary tree. There are several ways of precisely defining this notion, but applied to our case all of them imply, among other things, that the exponents $\delta$ and $2-\eta$ exist and take the values $\delta = 2$ and $2-\eta = 2 \wedge \alpha$. In a major breakthrough by Hara and Slade \cite{hara1990mean} mean-field behavior was established for high dimensional nearest-neighbour percolation. It was later also established for long-range percolation with $d > 6$ or $\alpha < \frac d3$ \cite{heydenreich2008mean}. The lower bounds in \Cref{theo:clustersize} rule out the mean-field behavior for $d \in \{1,2\}$ and $\alpha > \frac d3$, as they imply that $\delta>2$ in this regime.

\begin{theorem}\label{theo:twopointfct}
	Let $\alpha \in (0,1)$ for $d=1$, respectively $\alpha > 0$ for $d>1$. Suppose that $J(x,y) \leq C_1 \|x-y\|^{-d-\alpha}$
	and the exponent $2-\eta$ defined in \eqref{eq:twopointfct} exists. Then
	\begin{align*}
	2-\eta \geq \alpha\wedge 1\text.
	\end{align*}
\end{theorem}

\begin{figure}\label{fig:d=1}
	\begin{center}
		\begin{subfigure}[t]{7cm}
			\begin{tikzpicture}[xscale=3.5, yscale=3.75]
			\draw[->,thick, label=ab] (0,0) -- coordinate (x axis mid) (1.2,0);
			\draw[->,thick] (0,0) -- coordinate (y axis mid) (0,1.2);
			
			\draw[scale=1, domain=0:1, smooth, variable=\x, purple, very thick] plot ({\x}, {\x});

			\vertex[draw=none, label = $\alpha$ ] () at (1.15,-0.16) {};
			\vertex[draw=none, label = $2-\eta$ ] () at (-0.2,1) {};

			\foreach \x in {0,0.5,1}
			\draw (\x,1pt) -- (\x,-1pt)
			node[anchor=north] {\x};
			
			\foreach \y in {0,0.5,1}
			\draw (1pt,\y) -- (-1pt,\y) 
			node[anchor=east] {\y};
			\end{tikzpicture}
		\end{subfigure}
		\begin{subfigure}[t]{7cm}
			\begin{tikzpicture}[xscale=3.5, yscale=0.18]
			\draw[->,thick, label=ab] (0,0) -- coordinate (x axis mid) (1.2,0);
			\draw[->,thick] (0,0) -- coordinate (y axis mid) (0,25);
			
			\draw[scale=1, domain=0:1/3, smooth, variable=\x, blue, very thick] plot ({\x}, {(1+\x)/(1-\x)});
			
			\draw[scale=1, domain=0:0.92, smooth, variable=\x, yellow, very thick] plot ({\x}, {(2)/(1-\x)});
			
			\draw[scale=1, domain=0:1/3, smooth, variable=\x, red, very thick] plot ({\x}, {2});
			\draw[scale=1, domain=1/3:0.92, smooth, variable=\x, purple, very thick] plot ({\x}, {(1+\x)/(1-\x)});
			
			\vertex[draw=none, label = $\alpha$ ] () at (1.15,-3.5) {};
			\vertex[draw=none, label = $\delta$ ] () at (-0.1,22) {};

			\foreach \x in {0,0.5,1}
			\draw (\x,20pt) -- (\x,-20pt)
			node[anchor=north] {\x};
			
			\foreach \y in {0,10,20}
			\draw (1pt,\y) -- (-1pt,\y) 
			node[anchor=east] {\y};
			\end{tikzpicture}
		\end{subfigure}
		\parbox{14cm}{ \caption{ The critical exponents $2-\eta$ and $\delta$ for $d=1$. On the left: The purple line is the conjectured true value, our lower bound, and the upper bound proven in \cite{hutchcroft2022sharp}. On the right: The yellow curve is the upper bound on $\delta$ shown in \cite{hutchcroft2022sharp}, the red curve is the conjectured true value of $\delta$, and the blue curve is our lower bound. The part where the lower bound and the conjectured true value agree $\left(\alpha \in \left[\frac{1}{3},1\right)\right)$ is purple.}}
	\end{center}
\end{figure}
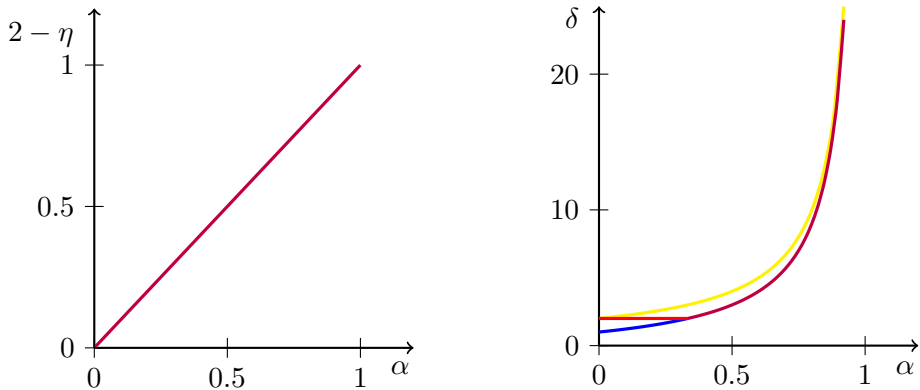

A graphical representation of our results, previously known results, and the conjectured behavior can be found in Figure 1 for dimension $d=1$ and in Figure 2 for dimension $d=2$ below.
Theorem \ref{theo:twopointfct} is an immediate consequence of \Cref{propo:two point decay}. In the case where $J(x,y) \simeq \|x-y\|^{-d-\alpha}$, \Cref{theo:twopointfct} shows together with Hutchcroft's result \cite{hutchcroft2022sharp} that $2-\eta = \alpha$ for $\alpha \leq 1$, respectively $\alpha <1$ for $d=1$, provided the exponent $2-\eta$ defined in \eqref{eq:twopointfct} exists. This also gives a partial solution to \cite[Problem 4.1]{hutchcroft2022sharp}, which asks for conditions under which the upper bound $2-\eta \leq \alpha$ has a matching lower bound. Provided that the conjectured picture described in \eqref{eq:crossover} below holds, our proof also shows that the {\sl crossover value} $\alpha_c(d)$ defined in \eqref{eq:crossover} below satisfies $\alpha_c(d)\geq 1$ for all dimensions $d\geq 2$.
We could alternatively define the exponent $2-\eta$ by $\sum_{x \in \Lambda_{n}} \pc \left( 0 \leftrightarrow x \right) \approx n^{2-\eta}$. For $\alpha <1$ the results of \cite{hutchcroft2022sharp} together with \Cref{propo:two point decay} show that the exponent $2-\eta$ defined like this exists and equals $\alpha$. See also the discussion after \Cref{propo:two point decay} for more details.\\

\noindent
Our proofs only assume an upper bound on the kernel $J$, so in particular the results are still valid for nearest-neighbor percolation. However, the bound $2-\eta \geq 1$ observed in this situation already follows from the proof of sharpness of the phase transition of Duminil-Copin and Tassion \eqref{eq:duminilcopin tassion}, and the lower bound $\delta \geq 3$ observed for $d=2$ follows from $2-\eta \geq 1$ and the hyperscaling inequality $(2-\eta)(\delta+1)\leq d (\delta-1)$ proven by Hutchcroft \cite{hutchcroft2021power}. This hyperscaling inequality can be rearranged to $\delta \geq \frac{d+2-\eta}{d-(2-\eta)}$ and using $d=2, 2-\eta \geq 1$ shows $\delta \geq 3$. But our proof still shows $\delta \geq 3$ without this machinery and without assuming the existence of the exponent $2-\eta$.
Our main tool for the proofs of \Cref{theo:clustersize} and \Cref{theo:twopointfct} (respectively \Cref{propo:cluster decay} and \Cref{propo:two point decay}) is a connection between the critical exponents and the isoperimetry of the boxes $\Lambda_n = \{-n,\ldots,n\}^d$ in \cref{sec:isoper}.

\begin{figure}\label{fig:d=2}
	\begin{center}
		\begin{subfigure}[t]{7cm}
			\begin{tikzpicture}[xscale=1.5, yscale=2]
			\draw[->,thick, label=ab] (0,0) -- coordinate (x axis mid) (3.2,0);
			\draw[->,thick] (0,0) -- coordinate (y axis mid) (0,2.2);
			
			\draw[scale=1, domain=0:1, smooth, variable=\x, purple, very thick] plot ({\x}, {\x});
			\draw[scale=1, domain=1:2.9, smooth, variable=\x, blue, very thick] plot ({\x}, {1});
			\draw[scale=1, domain=1:43/24, smooth, variable=\x, orange, very thick] plot ({\x}, {\x});
			\draw[scale=1, domain=43/24:2, smooth, variable=\x, yellow, very thick] plot ({\x}, {\x});
			\draw[scale=1, domain=43/24:2.9, smooth, variable=\x, red, very thick] plot ({\x}, {43/24});

			\vertex[draw=none, label = $\alpha$ ] () at (3.25,-0.3) {};
			\vertex[draw=none, label = $2-\eta$ ] () at (-0.4,2.1) {};

			\foreach \x in {0,1,2,3}
			\draw (\x,2pt) -- (\x,-2pt)
			node[anchor=north] {\x};
			
			\foreach \y in {0,1,2}
			\draw (2pt,\y) -- (-2pt,\y) 
			node[anchor=east] {\y};
			\end{tikzpicture}
		\end{subfigure}
		\begin{subfigure}[t]{7cm}
			\begin{tikzpicture}[xscale=1.2, yscale=0.18]
			\draw[->,thick, label=ab] (0,0) -- coordinate (x axis mid) (3.5,0);
			\draw[->,thick] (0,0) -- coordinate (y axis mid) (0,25);
			
			\draw[scale=1, domain=0:2/3, smooth, variable=\x, blue, very thick] plot ({\x}, {(2+\x)/(2-\x)});
			\draw[scale=1, domain=1:3.3, smooth, variable=\x, blue, very thick] plot ({\x}, {3});
			
			\draw[scale=1, domain=0:1.84, smooth, variable=\x, yellow, very thick] plot ({\x}, {(4)/(2-\x)});
			
			\draw[scale=1, domain=0:2/3, smooth, variable=\x, red, very thick] plot ({\x}, {2});
			\draw[scale=1, domain=1:43/24, smooth, variable=\x, red, very thick] plot ({\x}, {(2+\x)/(2-\x)});
			\draw[scale=1, domain=43/24:3.2, smooth, variable=\x, red, very thick] plot ({\x}, {91/5});
			\draw[scale=1, domain=2/3:1, smooth, variable=\x, purple, very thick] plot ({\x}, {(2+\x)/(2-\x)});
			
			\vertex[draw=none, label = $\alpha$ ] () at (3.35,-3.5) {};
			\vertex[draw=none, label = $\delta$ ] () at (-0.2,22) {};

			\foreach \x in {0,1,2,3}
			\draw (\x,20pt) -- (\x,-20pt)
			node[anchor=north] {\x};

			\foreach \y in {0,10,20}
			\draw (3pt,\y) -- (-3pt,\y) 
			node[anchor=east] {\y};
			\end{tikzpicture}
		\end{subfigure}
		\parbox{14cm}{\caption{ The critical exponents $2-\eta$ and $\delta$ for $d=2$. On the left: The blue line is our lower bound, the yellow line is the upper bound proven in \cite{hutchcroft2022sharp}, and the red line is the conjectured true value. The part where all three of them agree $\left(\alpha \in \left(0,1\right]\right)$ is purple and the part where the upper bound and the conjectured true value agree $\left(\alpha \in \left( 1, \frac{43}{24} \right]\right)$ is orange.
				On the right: The yellow curve is the upper bound on $\delta$ shown in \cite{hutchcroft2022sharp}, the red curve is the conjectured true value of $\delta$, and the blue curve is our lower bound. The part where the lower bound and the conjectured true value agree $\left(\alpha \in \left[\frac{2}{3},1\right]\right)$ is purple.}}
	\end{center}
\end{figure}
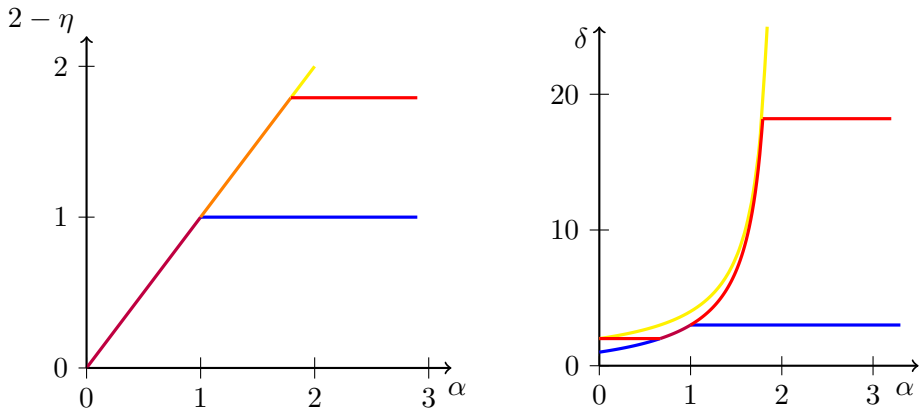

\paragraph*{Related work}

The critical behavior of percolating systems is typically a difficult problem. 
There has been considerable progress on the understanding of percolation on various graphs at and near criticality over the last years, see for example \cite{damron2017chemical,duminil2016absence,duminil2016new,duminil2017new,hermon2021no,hermon2021supercritical,hutchcroft2016critical,hutchcroft2020locality,hutchcroft2020new,hutchcroft2021critical,hutchcroft2021power,hutchcroft2022sharp,nachmias2008critical, drewitz2021critical}.
The physics prediction for the critical exponent $2-\eta$ is given by
\begin{align*}
2-\eta(d,\alpha) = \begin{cases}
\alpha & \text{ for }\alpha \leq 2-\eta_{\text{SR}}(d)\\
2-\eta_{\text{SR}} &  \text{ for }\alpha > 2-\eta_{\text{SR}}(d)
\end{cases}
\end{align*}
where $2-\eta_{\text{SR}}(d)$ is the corresponding exponent for short-range percolation on $\Z^d$. The prediction for the exponent $\delta$ is given by
\begin{align}\label{eq:crossover}
\delta(d,\alpha) = \begin{cases}
2 & \text{ for }\alpha \leq \frac{d}{3}\\
\frac{d+\alpha}{d-\alpha} &  \text{ for }\alpha \in \left[ \frac{d}{3}, \alpha_c (d) \right]\\
\delta_{\text{SR}}(d) & \text{ for } \alpha \geq \alpha_c(d)
\end{cases}
\end{align}
where $\delta_{\text{SR}}(d)$ is the corresponding exponent for short-range percolation and $\delta_{\text{SR}}(d)$ and the {\sl crossover value} $\alpha_c(d)$ are such that the function $\delta(d,\alpha)$ is continuous in $\alpha$. See also \cite[section 1.3]{hutchcroft2021power} or \cite[section 9 and 10]{grimmett1999percolation} for a broader overview of these predictions and references to the physics literature. The critical exponents are typically better understood in high dimension or for $\alpha < \frac{d}{3}$, where the triangle condition holds and methods involving the lace expansion can be used \cite{barsky1991percolation,borgs2005random,chen2015critical,hara1990mean,heydenreich2008mean}. Also for dimension $d=2$, and in particular for the triangular lattice, the situation is much better understood, due to works of Kesten, Smirnov and Werner \cite{smirnov2001,lawler2002one,smirnov2001critical,kesten1987scaling}. Here one knows that $\delta_{\text{SR}}(2)= \frac{91}{5}$. This also explains the conjectured pictures in Figure 2 and shows that the crossover value $\alpha_c(2)$ is expected to be $\frac{43}{24}$. Also for the hierarchical lattice the phase transition is better understood, due to recent results of Hutchcroft \cite{hutchcroft2021critical}. The lower bound $\delta \geq \frac{d+\alpha}{d-\alpha}$ proven for the hierarchical lattice is similar to our lower bound for $d=1$ and also shows absence of mean-field behavior for $\alpha > \frac{d}{3}$ on the hierarchical lattice.\\

\noindent
\textbf{Acknowledgements}
This work is supported by TopMath, the graduate program of the Elite Network of Bavaria and the graduate center of TUM Graduate School. We thank an anonymous referee for useful comments.

\section{Proofs}

Before going to the proofs, we want to introduce a theorem
that deals with the universal tightness of the maximum open cluster inside a random graph. It is a subset of \cite[Theorem 2.2]{hutchcroft2021power}, which turned out to be extremely useful in various models of random graphs. We write $\left|\Km(\Lambda)\right|$ for the cardinality of the largest open cluster in $\Lambda$. Note that $\Km(\Lambda)$ is in general not well-defined as a subset of $\Lambda$, since there can be distinct clusters with the same cardinality. But this will not cause any problems in the following. We define the typical value of $\left|\Km(\Lambda)\right|$ by
\begin{equation}\label{eq:typical value}
M_\beta (\Lambda) = \min \left\{ n \geq  0 : \p_\beta \left(|\Km(\Lambda)| \geq n \right) \leq e^{-1} \right\} \text.
\end{equation}
The theorem deals with general {\sl weighted graphs} $G = (V, E, J)$, where $J:E\to \left[0,\infty\right)$ is a function that gives weights to the edges. Now edges are open or closed independent of each other and an edge $e \in E$ is open with probability $1-e^{-\beta J(e)}$, where $\beta \geq 0$ is a parameter. In particular, long-range percolation on the integer lattice can be modelled as a weighted random graph with the weight function $J(\{x,y\}) = J(x-y)$.

\begin{theorem}[Universal tightness of the maximum cluster size]\label{theo:universaltightness}
	Let $G = (V, E, J)$ be a countable
	weighted graph and let $\Lambda \subseteq V$ be finite and non-empty. 
	Then the inequalities
	\begin{align}\label{eq:largest cluster tightness}
	&\p_\beta \left( \left|\Km(\Lambda)\right| \geq \alpha M_\beta (\Lambda) \right) \leq e^{-\frac{\alpha}{9}}\\
	\text{and } 
	& \label{eq:single cluster tightness}
	\p_\beta \left( \left|K_u \cap \Lambda \right| \geq \alpha M_\beta (\Lambda) \right) \leq
	e \cdot \p_\beta \left( \left|K_u \cap \Lambda \right| \geq  M_\beta (\Lambda) \right) e^{-\frac{\alpha}{9}}
	\end{align}
	hold for every $\beta \geq 0, \alpha \geq 1$, and $u\in V$.
\end{theorem}

We will use this theorem at many points in this paper.
For the lower bound on $\delta$ we define $\theta \coloneqq \frac{1}{\delta}$. In the following we will always assume that
\begin{equation}\label{eq:theta decay}
\sum_{k=1}^{n}\p_\beta \left( |K_0| \geq k \right) \leq C n^{1-\theta}
\end{equation}
holds for some constant $C< \infty$. Note that this already implies that $\p_\beta \left(|K_0|\geq n\right) \leq n^{-1} \sum_{k=1}^{n} \p_\beta \left(|K_0|\geq k\right)  \leq C n^{-\theta}$. Furthermore, for $\theta < 1$ the bound $\p_\beta \left( |K_0| \geq k \right) \leq C k^{-\theta}$ for all $k \in\{1,\ldots, n\}$ also implies \eqref{eq:theta decay} with a different constant $C^\prime$ depending on $C$ and $\theta$.

For the lower bound on the exponent of the two-point function $2-\eta$ we define $\Lambda_n = \{-n,\ldots,n\}^d$ and assume that
\begin{equation}\label{eq:eta decay}
\frac{1}{|\Lambda_n|} \sum_{x \in \Lambda_n} \p_\beta \left(0 \leftrightarrow x\right) \leq C n^{-d+2-\eta}
\end{equation}
holds for some constant $C< \infty$. From this definition we directly see that we can always assume that $-d+2-\eta \leq 0$, as the statement is trivially true otherwise.

\subsection{Moments of the cluster size inside boxes}

In this section, we give bounds on the expected size of the cluster inside boxes, i.e., $\E_\beta\left[ \left|K_0(\Lambda_n)\right| \right]$, given the upper bounds on the tail of the cluster \eqref{eq:theta decay} or the two-point function \eqref{eq:eta decay}. For $\Lambda \subset \Z^d$ and $x\in \Lambda$ we use the notation $K_x(\Lambda)$ for the set of vertices $y \in \Lambda$ that are connected to $x$ through an open path that lies entirely within $\Lambda$.
The next lemma translates bounds of the tail of the cluster size  into bounds of the typical largest cluster inside boxes of size $n$. The proof of such a statement has already been done for many different models of percolation. We give a short proof for completeness.

\begin{lemma}\label{lem:max cluster quantile bound}
	Assume that
	\eqref{eq:theta decay} holds
	for some constant $1\leq C<\infty$. Let $\Lambda\subset \Z^d$ be a finite set of size $n$.
	Then one has
	\begin{equation}\label{eq:max cluster quantile bound}
	M_\beta (\Lambda) \leq 3C n^\frac{1}{1+\theta}
	\end{equation}
\end{lemma}

\begin{proof}
	For $x\in \Lambda$, let $K_x(\Lambda)$ be the cluster of $x$ inside $\Lambda$. We use the notation $\tilde{C}=3C$ and get that
	\begin{align*}
	&\E_\beta \left[ \left| \left\{x \in \Lambda : |K_x(\Lambda)|\geq \tilde{C} n^\frac{1}{1+\theta} \right\} \right| \right]
	=
	\sum_{x \in \Lambda} \p_\beta \left( |K_x(\Lambda)| \geq \tilde{C} n^\frac{1}{1+\theta} \right)\\
	&
	\leq
	\sum_{x \in \Lambda} \p_\beta \left( |K_x| \geq \tilde{C} n^\frac{1}{1+\theta} \right)
	\leq \sum_{x \in \Lambda} C \tilde{C}^{-\theta} n^{-\frac{\theta}{1+\theta}}
	= C \tilde{C}^{-\theta} n n^{-\frac{\theta}{1+\theta}}
	= C \tilde{C}^{-\theta}  n^{\frac{1}{1+\theta}}\text.
	\end{align*}
	If there is one $x\in \Lambda$ such that $|K_x(\Lambda)|\geq \tilde{C} n^\frac{1}{1+\theta}$, then there are at least $\tilde{C} n^\frac{1}{1+\theta}$ many such $x\in \Lambda$. So in particular, if $\left|\Km(\Lambda)\right| \geq \tilde{C} n^{\frac{1}{1+\theta}}$, then there are at least $\tilde{C} n^{\frac{1}{1+\theta}}$ many vertices $x\in \Lambda$ with $|K_x(\Lambda)|\geq \tilde{C} n^\frac{1}{1+\theta}$. This implies that
	\begin{align*}
	\mathbbm{1}_{\left\{\left|\Km(\Lambda)\right| \geq \tilde{C} n^{\frac{1}{1+\theta}}\right\}} 
	\leq 
	\frac{1}{\tilde{C} n^\frac{1}{1+\theta}} \left| \left\{x \in \Lambda : |K_x(\Lambda)|\geq \tilde{C} n^\frac{1}{1+\theta} \right\} \right|
	\end{align*}
	and taking expectations on both sides yields that
	\begin{align*}
	& \p_\beta \left( \left|\Km(\Lambda)\right| \geq \tilde{C} n^{\frac{1}{1+\theta}} \right)
	\leq
	\frac{1}{\tilde{C} n^\frac{1}{1+\theta}} \E_\beta \left[\left| \left\{x \in \Lambda : |K_x(\Lambda)|\geq \tilde{C} n^\frac{1}{1+\theta} \right\} \right|\right]\\
	& \leq \frac{1}{\tilde{C} n^\frac{1}{1+\theta}} C \tilde{C}^{-\theta}  n^{\frac{1}{1+\theta}} = C \tilde{C}^{-1-\theta} = C (3C)^{-1-\theta} < \frac{1}{3} < \frac{1}{e}
	\end{align*}
	which shows that $M_\beta(\Lambda) \leq 3Cn^\frac{1}{1+\theta}$.
	
\end{proof}

\begin{lemma}\label{lem:cluster in box moment bound}
	Assume that \eqref{eq:theta decay} holds. Let $\Lambda\subset \Z^d$ be a finite set of size $n$.
	Then there exists a constant $C_2=C_2(C,\theta)$ such that
	\begin{equation}\label{eq:cluster in box moment bound}
	\E_\beta\left[\left|K_0(\Lambda)\right|\right] \leq C_2 n^{\frac{1-\theta}{1+\theta}}.
	\end{equation}
\end{lemma}

\begin{proof}
	The proof is heavily based on the use of \Cref{theo:universaltightness}. For abbreviation, we simply write $M=M_\beta(\Lambda)$. Thus we get that
	\begin{align*}
	&\E_\beta\left[\left|K_0(\Lambda)\right|\right] = 
	\sum_{k=1}^{\infty} \p_\beta \left( \left|K_0(\Lambda)\right| \geq k \right)
	= \sum_{l=0}^{\infty} \sum_{k=1}^{M}
	\p_\beta \left( \left|K_0(\Lambda)\right| \geq l M + k  \right) \\
	& = \sum_{k=1}^{M}
	\p_\beta \left( \left|K_0(\Lambda)\right| \geq  k \right)
	+
	\sum_{l=1}^{\infty} \sum_{k=1}^{M}
	\p_\beta \left( \left|K_x(\Lambda)\right| \geq l M + k  \right)\\
	&
	\leq
	C M^{1-\theta}
	+
	\sum_{l=1}^{\infty} \sum_{k=1}^{M}
	\p_\beta \left( \left|K_0(\Lambda)\right| \geq l M \right)\\
	&
	\overset{\eqref{eq:single cluster tightness}}{\leq}
	C M^{1-\theta}
	+
	M
	\sum_{l=1}^{\infty} 
	e
	\p_\beta \left( \left|K_0(\Lambda)\right| \geq  M \right) e^{-\frac{l}{9}}\\
	&
	\leq C M^{1-\theta} + e  C M^{1-\theta} \sum_{l=1}^{\infty} 
	e^{-\frac{l}{9}}
	\leq C^{\prime } M^{1-\theta} \leq C_2 n^{\frac{1-\theta}{1+\theta}}
	\end{align*}
	for some constants $C^\prime, C_2 < \infty$. Here we used the result of \Cref{lem:max cluster quantile bound} for the last inequality.
\end{proof}

The next Lemma translates the average bound on the two-point function \eqref{eq:eta decay} into bounds on the restricted cluster size.  For two sets $A,B\subset\Z^d$ we introduce the notation $A \overset{\Lambda_n}{\longleftrightarrow} B$, meaning that there exists a path from $A$ to $B$ that uses edges with both endpoints in $\Lambda_n$ only.

\begin{lemma}
	Assume that \eqref{eq:eta decay} holds. Then one has
	\begin{align*}
	\E_\beta \left[\left|K_0(\Lambda_n)\right|\right] \leq 3^d C n^{2-\eta}.
	\end{align*}
	for all $x \in \Lambda_n$.
\end{lemma}
\begin{proof}
	The $\infty$-distance between different $0$ and $x \in \Lambda_n$ is at most $n$. We have that $\left|\Lambda_n\right|= (2n+1)^d$. Thus linearity of expectation gives that
	\begin{align*}
	&\E_\beta \left[\left|K_0(\Lambda_n)\right|\right] = \sum_{x \in \Lambda_n} \p_\beta \left(0 \overset{\Lambda_n}{\longleftrightarrow} x \right)
	\leq
	\left|\Lambda_{n}\right| \frac{1}{\left|\Lambda_{n}\right|}
	\sum_{x \in \Lambda_{n}} \p_\beta \left(0 \leftrightarrow x \right)\\
	& \leq (2n+1)^d C n^{-d+2-\eta} \leq 3^d C n^{2-\eta}.
	\end{align*}
\end{proof}

\subsection{Isoperimetric inequalities in expectation}\label{sec:isoper}

In this section, we use the isoperimetry of the box $\Lambda_n = \{-n,\ldots, n\}^d$ in order to bound the expected number of edges at the boundary of the box, for which the end inside the box is connected to $0$.  For long-range percolation with a kernel $J : \Z^d\times \Z^d \rightarrow \left[0,\infty\right)$ satisfying $J(x,y) \simeq \|x-y\|^{-d-\alpha}$ the isoperimetry of the box $\Lambda_n$ changes at $\alpha =1$. More precisely, if we denote by $\partial \Lambda_n$ the set of open edges with exactly one endpoint in $\Lambda_n$, we have that
\begin{align*}
\E_\beta\left[ |\partial \Lambda_n | \right] \simeq \begin{cases}
n^{d-\alpha} & \text{ if } \alpha <1\\
n^{d-1}\log(n) & \text{ if } \alpha =1\\
n^{d-1} & \text{ if } \alpha > 1\\
\end{cases}.
\end{align*}
Consequently, we see that for $\alpha<1$ long-range effects determine the isoperimetry of the box, whereas for $\alpha\geq 1$ the short-range effects dominate, with logarithmic corrections at $\alpha = 1$. In particular, a point $x\in \Lambda_n$ that is chosen uniformly at random will have of order $n^{-(\alpha\wedge 1)+o(1)}$ neighbors outside of the box. This is also the reason, why the term $\alpha \wedge 1$ pops up in the statements of \Cref{theo:clustersize} and \Cref{theo:twopointfct}. In the following, for two sets $A,B\subset \Z^d$ we use the notation $A\sim B$ if there exists a direct edge from $A$ to $B$. We also use a statement that was shown by Duminil-Copin and Tassion in \cite{duminil2016new,duminil2017new}. There it is shown that for $\beta \geq \beta_{c}$ and all finite sets $S\subset \Z^d$ containing the origin $0$ one has
\begin{align}\label{eq:duminilcopin tassion}
& \phi_\beta \left( S \right) \coloneqq
\sum_{x \in S} \sum_{y \notin S}
\left(1-e^{-\beta J(x,y)}\right) \p_\beta \left( 0 \overset{S}{\longleftrightarrow} x\right)
\geq  1\text.
\end{align}
Moreover, they also showed the reverse direction, i.e., that $\phi_\beta(S)\geq 1$ for all finite sets $S\subset \Z^d$ with $0\in S$ implies $\beta \geq \beta_c$, but we will not use this statement in our proof. Similar results to the result in \eqref{eq:duminilcopin tassion} were already shown previously, see for example \cite[Lemma 3.1]{kozma2011arm} or \cite[Lemma 5.1]{aizenman1986discontinuity}.

\begin{lemma}\label{lem:cluster outgoing}
	We write $K_0(\Lambda_k)$ for the set of vertices $y \in \Lambda_k$ that are connected to $0$ through an open path that lies entirely within $\Lambda_k$. Let $n\in \N$ be arbitrary and fixed. For $d=1$ and all $\alpha \in \left(0,1\right)$, respectively for $d>1$ and all $\alpha > 0$, and all $\beta > 0$, there exists a constant $C_3=C_3(\alpha,\beta, d)$ that does not depend on $n$, so that there exists a $k \in \{1,\ldots, n\}$ with
	\begin{align}\label{eq:cluster outgoing d=1}
	& \phi_\beta \left( \Lambda_k \right) =
	\sum_{x \in \Lambda_k} \sum_{y \notin \Lambda_k}
	\left(1-e^{-\beta J(x,y)}\right) \p_\beta \left( 0 \overset{\Lambda_k}{\longleftrightarrow} x\right)
	\leq C_3  \E_\beta\left[ \left| K_0\left(\Lambda_{n}\right) \right| \right]f(n,\alpha) 
	\end{align}
	where the function $f(n,\alpha)$ is defined by
	\begin{equation}\label{eq:f}
	f(n,\alpha) = \begin{cases}
	n^{-\alpha} & \text{ if } \alpha < 1 \\
	n^{-1}\log(n) & \text{ if } \alpha = 1 \\
	n^{-1} & \text{ if } \alpha > 1 \\
	\end{cases}.
	\end{equation}
\end{lemma}

\begin{proof}
	For $x \in \Lambda_{n}$ we write $t_x \coloneqq \p_\beta\left(x \overset{\Lambda_{n}}{\longleftrightarrow} 0\right)$ and get that
	\begin{align}\label{eq:X_k bound 0}
	\sum_{x \in \Lambda_{n}} t_x & 
	=
	\sum_{x \in \Lambda_{n}} \p_\beta\left(x \overset{\Lambda_{n}}{\longleftrightarrow} 0\right)
	=
	\E_\beta\left[\left| K_0(\Lambda_{n})\right| \right] \text.
	\end{align}
	Next, we define $X_k$ as the number of open edges between $\Lambda_k$ and $(\Lambda_k)^C$ for which one end is connected to $0$ within $\Lambda_k$. Formally, we define
	\begin{align*}
	X_k \coloneqq \left| \left\{ e=\{a,b\} \text{ open} : a \in \Lambda_k, b \notin \Lambda_k  , \text{ and } 0 \overset{\Lambda_k}{\longleftrightarrow} a \right\} \right|\text.
	\end{align*}
	The occupation status of edges inside $\Lambda_k$ and of edges with one end outside of $\Lambda_k$ are independent random variables. So by linearity of expectation one has
	\begin{align*}
	\E_\beta \left[X_k\right] =
	\sum_{a \in \Lambda_k} \sum_{b \notin \Lambda_k}
	\left(1-e^{-\beta J(a,b)}\right) \p_\beta \left( 0 \overset{\Lambda_k}{\longleftrightarrow} a\right)
	=
	\phi_\beta \left(\Lambda_k\right)\text.
	\end{align*}
	Thus, it suffices to bound the expected value of $X_k$ and show that there exists a $k \in \{1,\ldots, n\}$ such that the expected value $\E_\beta \left[X_k\right]$ is reasonably small, as in \eqref{eq:cluster outgoing d=1}.
	For this, let $K$ be a random variable that is uniformly distributed on $\{1,\ldots, n\}$ and is independent of the percolation configuration. We write $\mathbf{P}_\beta$ for the joint distribution of the percolation configuration and $K$, and $\mathbf{E}_\beta$ for its expectation. Thus we get
	\begin{align} \notag
	& \mathbf{E}_\beta \left[X_K\right] =
	\mathbf{E}_\beta \left[
	\left| \left\{ \{a,b\} \text{ open} : a \in \{-K,\ldots,K\}^d, b \notin \{-K,\ldots,K\}^d  , \text{ and } 0 \overset{\Lambda_K}{\longleftrightarrow} a \right\} \right| \right]\\
	& \notag = \frac{1}{n} \sum_{k=1}^{n}
	\E_\beta \left[
	\left| \left\{ \{a,b\} \text{ open} : a \in \{-k,\ldots,k\}^d, b \notin \{-k,\ldots,k\}^d  , \text{ and } 0 \overset{\Lambda_k}{\longleftrightarrow} a \right\} \right| \right]\\
	& \label{eq:X_K expectation bound 1} = \frac{1}{n} \sum_{k=1}^{n} \sum_{a \in \Lambda_{n}} \sum_{b \in \Z^d} \E_\beta \left[ \mathbbm{1}_{\{a \in \Lambda_k\}} \mathbbm{1}_{\{b \notin \Lambda_k\}} \mathbbm{1}_{\left\{ 0 \overset{\Lambda_k}{\longleftrightarrow} a \right\}} \mathbbm{1}_{\{a \sim b\}} \right] \text.
	\end{align}
	For fixed $k$, the events $\{0 \overset{\Lambda_k}{\longleftrightarrow} a \}$ and $\{\{a,b\} \text{ is open}\}$ are independent for $b\notin \Lambda_k$, as the first event depends only on edges with both endpoints inside $\Lambda_k$. For fixed $a\in\Lambda_{n}$, the expression $\p_\beta \left(0\overset{\Lambda_k}{\longleftrightarrow} a\right)$ can only be positive if $k \geq \|a\|_\infty$. Combining the two previous observations we get that
	\begin{align}
	&\notag \mathbf{E}_\beta \left[X_K\right] = \frac{1}{n} \sum_{k=1}^{n} \sum_{a \in \Lambda_{k}} \sum_{b \in \Z^d\setminus \Lambda_k} \p_\beta \left(0 \overset{\Lambda_k}{\longleftrightarrow} a \right)  \p_\beta \left(a \sim b\right)\\
	& \notag = \frac{1}{n} \sum_{a \in \Lambda_{n}} \sum_{k=1\vee \|a\|_\infty}^{n}  \sum_{b \in \Z^d\setminus \Lambda_k} \p_\beta \left(0 \overset{\Lambda_k}{\longleftrightarrow} a \right)  \p_\beta \left(a \sim b\right)\\
	& \notag
	\leq
	\sum_{a \in \Lambda_{n}} \p_\beta \left( 0 \overset{\Lambda_n}{\longleftrightarrow} a \right) \left(\frac{1}{n} \sum_{k=1 \vee \|a\|_\infty}^{n}  \ \sum_{b \in \Z^d\setminus \Lambda_k} \left( 1-e^{-\beta J(a,b)} \right) \right)\\
	& \label{eq:X_K expectation bound 2}
	\leq 
	\sum_{a \in \Lambda_{n}} t_a \left(\frac{1}{n} \sum_{k=1 \vee \|a\|_\infty}^{n}  \ \sum_{b \in \Z^d\setminus \Lambda_k} \beta C_1 \|a-b\|^{-d-\alpha}\right)\text,
	\end{align}
	where we used that $1-e^{-x}\leq x$ for the last inequality.
	Now, for fixed $a \in \Lambda_{n}$ and $k\geq \|a\|_\infty$ there exist constants $C_1^\prime = C_1^\prime(C_1,d,\beta) < \infty$ and $ C_1^{\prime\prime} = C_1^{\prime\prime}(C_1,d,\alpha,\beta) < \infty$ such that
	\begin{align}
	\notag &\sum_{b \in \Z^d\setminus \Lambda_k} \beta C_1 \|a-b\|^{-d-\alpha} 
	\leq 
	\sum_{l=k+1 - \|a\|_\infty}^{\infty} \sum_{b \in \Z^d : \|b-a\|_\infty=l} \beta C_1 \|a-b\|^{-d-\alpha}\\
	& \notag 
	=
	\sum_{l=k+1 - \|a\|_\infty}^{\infty} \sum_{b \in \Z^d : \|b\|_\infty=l} \beta C_1 \|b\|^{-d-\alpha}
	\leq
	\sum_{l=k+1 - \|a\|_\infty}^{\infty} C_1^\prime l^{d-1} l^{-d-\alpha}\\
	& \label{eq:X_k bound 3}
	= C_1^\prime \sum_{l=k+1 - \|a\|_\infty}^{\infty}  l^{-1-\alpha} \leq C_1^{\prime \prime} (k+1-\|a\|_\infty)^{-\alpha}.
	\end{align}
	Using \eqref{eq:X_k bound 3} we see that
	\begin{align}
	& \notag \frac{1}{n} \sum_{k=1 \vee \|a\|_\infty}^{n}  \ \sum_{b \in \Z^d\setminus \Lambda_k} \beta C_1 \|a-b\|^{-d-\alpha} 
	\leq 
	\frac{1}{n} \sum_{k= \|a\|_\infty}^{n} C_1^{\prime \prime} (k+1-\|a\|_\infty)^{-\alpha}\\
	&
	\leq
	C_1^{\prime \prime} \frac{1}{n} \sum_{k=1}^{n+1} k^{-\alpha}
	\leq \hat{C_1} f(n,\alpha)
	\end{align}
	for a constant $\hat{C_1} = \hat{C_1}(C_1^{\prime\prime},\alpha)< \infty$. Inserting this result into \eqref{eq:X_K expectation bound 2} yields
	\begin{align*}
	\frac{1}{n} \sum_{k=1}^{n} \E_\beta\left[X_k\right]
	=
	\mathbf{E}_\beta\left[X_K\right] 
	\leq \sum_{a \in \Lambda_{n}} t_a \hat{C_1} f(n,\alpha) \overset{\eqref{eq:X_k bound 0}}{= } \E_\beta \left[ \left|K_0\left(\Lambda_{n}\right)\right| \right] \hat{C_1}   f(n,\alpha)\text.
	\end{align*}
	So in particular there needs to exist at least one $k \in \{1,\ldots, n\}$ for which $\E_\beta\left[X_k\right] \leq \E_\beta \left[ \left|K_0\left(\Lambda_{n}\right)\right| \right] \hat{C_1}  f(n,\alpha)$, which finishes the proof.
\end{proof}

\subsection{The proof of \Cref{theo:clustersize} and \Cref{theo:twopointfct}}

Now we are ready to go to the main proofs. \Cref{theo:clustersize} is an immediate consequence of \Cref{propo:cluster decay} and \Cref{theo:twopointfct} is an immediate consequence of \Cref{propo:two point decay}. Also remember the definition of the function $f$ defined in \eqref{eq:f} which we will use at several points below.

\begin{proposition}\label{propo:two point decay}
	Let $\alpha \in (0,1)$ for $d=1$, respectively $\alpha>0$ for $d>1$, and assume that there exists a constant $C_1<\infty$ such that $J(x,y) \leq C_1 \|x-y\|^{-d-\alpha}$ for all $x,y \in \Z^d$. Provided $\beta_c<\infty$ one has $ \sum_{x \in \Lambda_{n}} \pc \left( 0 \leftrightarrow x \right) \geq \frac{1}{C_3} f(n,\alpha)^{-1}$ where $C_3$ is the same constant as in \Cref{lem:cluster outgoing}.
\end{proposition}

\begin{proof}
	We will first show that $\Ec \left[|K_0\left(\Lambda_{n}\right)|\right] \geq \frac{1}{C_3} f(n,\alpha)^{-1}$.
	Assume the contrary, i.e., $\Ec \left[|K_0\left(\Lambda_{n}\right)|\right] < \frac{1}{C_3} f(n,\alpha)^{-1}$. Then by \Cref{lem:cluster outgoing} there exists a $k \in \{1,\ldots,n\}$ with 
	\begin{align*}
	\phi_{\beta_c} \left( \Lambda_k \right) \leq C_3 \Ec \left[\left|K_0 \left(\Lambda_{n}\right) \right| \right] f(n,\alpha) < 1
	\end{align*}
	which is a contradiction to \eqref{eq:duminilcopin tassion}.
	Now, by linearity of expectation we have that
	\begin{align}\label{eq:inequality}
	&\sum_{x \in \Lambda_{n}} \pc \left( 0 \leftrightarrow x \right)
	\geq
	\sum_{x \in \Lambda_{n}} \pc \left( 0 \overset{\Lambda_{n}}{\longleftrightarrow} x \right)
	=
	\Ec \left[\left|K_0\left(\Lambda_{n}\right)\right|\right]
	\geq
	\frac{1}{C_3}f(n,\alpha)^{-1}.
	\end{align}
\end{proof}
\noindent
\Cref{propo:two point decay} shows in particular that for a small enough constant $c>0$ we have
\begin{align*}
\frac{1}{|\Lambda_{n}|}\sum_{x \in \Lambda_{n}} \pc \left( 0 \leftrightarrow x \right)
\geq c n^{-d}f(n,\alpha)^{-1} = \begin{cases}
c n^{-d+\alpha} & \text{for } \alpha < 1 \\
c n^{-d+1} \log(n)^{-1} & \text{for } \alpha = 1 \\
c n^{-d+1} & \text{for } \alpha > 1 \\
\end{cases}
\end{align*}
which shows that the exponent $2-\eta$ defined in \eqref{eq:twopointfct} satisfies $2-\eta \geq \alpha \wedge 1$, provided the exponent $2-\eta$ exists. In \cite{hutchcroft2022sharp} it is shown that $\frac{1}{|\Lambda_{n}|}\sum_{x \in \Lambda_{n}} \pc \left( 0 \leftrightarrow x \right) = \mathcal{O} \left(n^{-d+\alpha}\right)$. Combining this with \Cref{propo:two point decay} we get that for $\alpha < 1$ and a kernel $J$ satisfying $J(x,y)\simeq \|x-y\|^{-d-\alpha}$ one has
\begin{align*}
\frac{1}{|\Lambda_{n}|}\sum_{x \in \Lambda_{n}} \pc \left( 0 \leftrightarrow x \right) \simeq n^{-d+\alpha}.
\end{align*}
So when we alternatively define the two-point critical exponent $2-\eta$ by the averaged version $\frac{1}{\left|\Lambda_n\right|}\sum_{x \in \Lambda_{n}} \pc \left( 0 \leftrightarrow x \right) \approx n^{-d+2-\eta}$, then we see that this exponent exists for $\alpha<1$ and equals $\alpha$. However, it is not clear whether this statements holds without averaging, i.e., if the exponent $2-\eta$ defined as in \eqref{eq:twopointfct} also exists. See also \cite[Problem 4.3]{hutchcroft2022sharp} for a related problem. Next, we consider the lower bound on the exponent $\delta$.

\begin{proposition}\label{propo:cluster decay}
	Let $\alpha \in (0,1)$ for $d=1$, respectively $\alpha>0$ for $d>1$, and assume that there exists a constant $C_1<\infty$ such that $J(x,y) \leq C_1 \|x-y\|^{-d-\alpha}$ for all $x,y \in \Z^d$. Suppose that $\beta_c<\infty$ and $\sum_{k=1}^n \pc \left( |K_0| \geq k \right) \leq C n^{1-\frac{1}{\delta}}$ for all $n\in \N$. Then $\delta \geq \frac{d+(\alpha\wedge 1)}{d-(\alpha\wedge 1)}$.
\end{proposition}

\begin{proof}
	We write $\theta = \frac{1}{\delta}$ and get that $\sum_{k=1}^{N} \pc \left( |K_0| \geq k \right) \leq C N^{1-\theta}$ for all $N \in \N$. \Cref{lem:cluster in box moment bound} shows that for some constant $C^\prime <\infty$ we have $\E_\beta \left[\left|  K_0\left(\Lambda_{n}\right) \right|\right] \leq C^\prime n^{d \frac{1-\theta}{1+\theta}}$. Combining this with inequality \eqref{eq:inequality} we get that
	\begin{align*}
	C^\prime n^{d \frac{1-\theta}{1+\theta}}
	\geq
	\E_\beta \left[\left|  K_0\left(\Lambda_{n}\right) \right|\right] \geq C_3^{-1} f(n,\alpha)^{-1} \approx n^{(\alpha \wedge 1) + o(1)}
	\end{align*} 
	and this shows that $d\frac{1-\theta}{1+\theta}\geq \alpha \wedge 1$.
	As we consider $\alpha \in (0,1)$ only for $d=1$, we always have that $\frac{\alpha \wedge 1}{d}<1$.
	Elementary calculations show that
	\begin{align*}
	& \ d \ \frac{1-\theta}{1+\theta} = d \ \frac{\delta-1}{\delta +1} \geq \alpha\wedge 1  
	\Leftrightarrow \delta-1 \geq \frac{\alpha \wedge 1}{d} \delta + \frac{\alpha \wedge 1}{d}\\
	\Leftrightarrow & \
	\delta - \frac{\alpha \wedge 1}{d}\delta = \delta \left(1- \frac{\alpha \wedge 1}{d}\right) \geq \frac{\alpha \wedge 1}{d} + 1\\
	\Leftrightarrow & \
	\delta \geq \frac{1+\frac{\alpha \wedge 1}{d}}{1-\frac{\alpha \wedge 1}{d}} = \frac{d+(\alpha \wedge 1)}{d-(\alpha \wedge 1)}
	\end{align*}
	which finishes the proof.
\end{proof}


\begin{thebibliography}{}
	
	\small
	
	\bibitem{aizenman1987sharpness}
	Michael Aizenman and David~J. Barsky.
	\newblock Sharpness of the phase transition in percolation models.
	\newblock {\em Communications in Mathematical Physics}, 108(3) (1987) 489--526.
	
	\bibitem{aizenman1986discontinuity}
	Michael Aizenman and Charles~M. Newman.
	\newblock Discontinuity of the percolation density in one dimensional $1/|x-y|^2$ percolation models.
	\newblock {\em Communications in Mathematical Physics}, 107(4) (1986) 611--647.
	
	\bibitem{baeumler2022behavior}
	Johannes B{\"a}umler.
	\newblock Behavior of the distance exponent for $\frac{1}{|x-y|^{2d}}$
	long-range percolation. (2022)
	\newblock {\em arXiv preprint arXiv:2208.04793}.
	
	\bibitem{baeumler2022distances}
	Johannes B{\"a}umler.
	\newblock Distances in $\frac{1}{|x-y|^{2d}}$ percolation models for all dimensions. (2022)
	\newblock {\em arXiv preprint arXiv:2208.04800}.
	
	\bibitem{barsky1991percolation}
	David~J. Barsky and Michael Aizenman.
	\newblock Percolation critical exponents under the triangle condition.
	\newblock {\em The Annals of Probability}, (1991) 1520--1536.
	
	\bibitem{benjamini2001diameter}
	Itai Benjamini and Noam Berger.
	\newblock The diameter of long-range percolation clusters on finite cycles.
	\newblock {\em Random Structures \& Algorithms}, 19(2) (2001) 102--111.
	
	\bibitem{berger2002transience}
	Noam Berger.
	\newblock Transience, recurrence and critical behavior for long-range
	percolation.
	\newblock {\em Communications in mathematical physics}, 226(3) (2002) 531--558.
	
	\bibitem{berger2004lower}
	Noam Berger.
	\newblock A lower bound for the chemical distance in sparse long-range
	percolation models. (2004)
	\newblock {\em arXiv preprint math/0409021}.
	
	\bibitem{biskup2004scaling}
	Marek Biskup.
	\newblock On the scaling of the chemical distance in long-range percolation
	models.
	\newblock {\em The Annals of Probability}, 32(4) (2004) 2938--2977.
	
	\bibitem{biskup2011graph}
	Marek Biskup.
	\newblock Graph diameter in long-range percolation.
	\newblock {\em Random Structures \& Algorithms}, 39(2) (2011) 210--227.
	
	\bibitem{borgs2005random}
	Christian Borgs, Jennifer~T. Chayes, Gordon Slade, Joel Spencer, and Remco
	van~der Hofstad.
	\newblock Random subgraphs of finite graphs. ii. the lace expansion and the
	triangle condition.
	\newblock {\em The Annals of Probability}, 33(5) (2005) 1886--1944.
	
	\bibitem{chen2015critical}
	Lung-Chi Chen and Akira Sakai.
	\newblock Critical two-point functions for long-range statistical-mechanical
	models in high dimensions.
	\newblock {\em The Annals of Probability}, 43(2) (2015) 639--681.
	
	\bibitem{damron2017chemical}
	Michael Damron, Jack Hanson, and Philippe Sosoe.
	\newblock On the chemical distance in critical percolation.
	\newblock {\em Electronic Journal of Probability}, 22 (2017) 1--43.
	
	\bibitem{ding2013distances}
	Jian Ding and Allan Sly.
	\newblock Distances in critical long range percolation. (2013)
	\newblock {\em arXiv preprint arXiv:1303.3995}.
	
	\bibitem{drewitz2021critical}
	Alexander Drewitz, Alexis Pr{\'e}vost, and Pierre-Fran{\c{c}}ois Rodriguez.
	\newblock Critical exponents for a percolation model on transient graphs. (2021)
	\newblock {\em arXiv preprint arXiv:2101.05801}.
	
	\bibitem{duminil2016absence}
	Hugo Duminil-Copin, Vladas Sidoravicius, and Vincent Tassion.
	\newblock Absence of infinite cluster for critical bernoulli percolation on
	slabs.
	\newblock {\em Communications on Pure and Applied Mathematics},
	69(7) (2016) 1397--1411.
	
	\bibitem{duminil2016new}
	Hugo Duminil-Copin and Vincent Tassion.
	\newblock A new proof of the sharpness of the phase transition for bernoulli
	percolation and the ising model.
	\newblock {\em Communications in Mathematical Physics}, 343(2) (2016) 725--745.
	
	\bibitem{duminil2017new}
	Hugo Duminil-Copin and Vincent Tassion.
	\newblock A new proof of the sharpness of the phase transition for bernoulli
	percolation on $\mathbb{Z}^d$.
	\newblock {\em L’Enseignement math{\'e}matique}, 62(1) (2017) 199--206.
	
	\bibitem{duminil2020long}
	Hugo Duminil-Copin, Christophe Garban, and Vincent Tassion.
	\newblock Long-range models in 1d revisited. (2020)
	\newblock {\em arXiv preprint arXiv:2011.04642}.
	
	\bibitem{grimmett1999percolation}
	Geoffrey Grimmett.
	\newblock Percolation, volume 321 of
	\newblock {\em Grundlehren der Mathematischen Wissenschaften [Fundamental
		Principles of Mathematical Sciences]}, 2. ed (1999), Berlin, Heidelberg.
	
	\bibitem{hara1990mean}
	Takashi Hara and Gordon Slade.
	\newblock Mean-field critical behaviour for percolation in high dimensions.
	\newblock {\em Communications in Mathematical Physics}, 128(2) (1990) 333--391.
	
	\bibitem{hermon2021no}
	Jonathan Hermon and Tom Hutchcroft.
	\newblock No percolation at criticality on certain groups of intermediate
	growth.
	\newblock {\em International Mathematics Research Notices},
	2021(22) (2021) 17433--17455.
	
	\bibitem{hermon2021supercritical}
	Jonathan Hermon and Tom Hutchcroft.
	\newblock Supercritical percolation on nonamenable graphs: Isoperimetry,
	analyticity, and exponential decay of the cluster size distribution.
	\newblock {\em Inventiones mathematicae}, 224(2) (2021) 445--486.
	
	\bibitem{heydenreich2008mean}
	Markus Heydenreich, Remco van~der Hofstad, and Akira Sakai.
	\newblock Mean-field behavior for long-and finite range ising model,
	percolation and self-avoiding walk.
	\newblock {\em Journal of Statistical Physics}, 132(6) (2008) 1001--1049.
	
	\bibitem{heydenreich2017progress}
	Markus Heydenreich and Remco Van~der Hofstad.
	\newblock {\em Progress in high-dimensional percolation and random graphs}, (2017)
	\newblock Springer, {\em CRM short courses}.
	
	\bibitem{hutchcroft2016critical}
	Tom Hutchcroft.
	\newblock Critical percolation on any quasi-transitive graph of exponential
	growth has no infinite clusters.
	\newblock {\em Comptes Rendus Mathematique}, 354(9) (2016) 944--947.
	
	\bibitem{hutchcroft2020locality}
	Tom Hutchcroft.
	\newblock Locality of the critical probability for transitive graphs of
	exponential growth.
	\newblock {\em The Annals of Probability}, 48(3) (2020) 1352--1371.
	
	\bibitem{hutchcroft2020new}
	Tom Hutchcroft.
	\newblock New critical exponent inequalities for percolation and the random
	cluster model.
	\newblock {\em Probability and Mathematical Physics}, 1(1) (2020) 147--165.
	
	\bibitem{hutchcroft2021power}
	Tom Hutchcroft.
	\newblock Power-law bounds for critical long-range percolation below the
	upper-critical dimension.
	\newblock {\em Probability Theory and Related Fields}, 181(1) (2021) 1--38.
	
	\bibitem{hutchcroft2021critical}
	Tom Hutchcroft.
	\newblock The critical two-point function for long-range percolation on the
	hierarchical lattice. (2021)
	\newblock {\em arXiv preprint arXiv:2103.17013}.
	
	\bibitem{hutchcroft2022sharp}
	Tom Hutchcroft.
	\newblock Sharp hierarchical upper bounds on the critical two-point function
	for long-range percolation on $\mathbb{Z}^d$. (2022)
	\newblock {\em arXiv preprint arXiv:2202.07634}.
	
	\bibitem{kesten1987scaling}
	Harry Kesten.
	\newblock Scaling relations for 2d-percolation.
	\newblock {\em Communications in Mathematical Physics}, 109(1) (1987) 109--156.
	
	\bibitem{kozma2011arm}
	Gady Kozma and Asaf Nachmias.
	\newblock Arm exponents in high dimensional percolation.
	\newblock {\em Journal of the American Mathematical Society},  24(2) (2011) 375--409.
	
	\bibitem{lawler2002one}
	Gregory Lawler, Oded Schramm, and Wendelin Werner.
	\newblock One-arm exponent for critical 2d percolation.
	\newblock {\em Electronic Journal of Probability}, 7 (2002) 1--13.
	
	\bibitem{nachmias2008critical}
	Asaf Nachmias and Yuval Peres.
	\newblock Critical random graphs: diameter and mixing time.
	\newblock {\em The Annals of Probability}, 36(4) (2008) 1267--1286.
	
	\bibitem{newman1986one}
	Charles~M. Newman and Lawrence~S. Schulman.
	\newblock One dimensional $1/|j-i|^s$ percolation models: The existence of a
	transition for $s\leq 2$.
	\newblock {\em Communications in Mathematical Physics}, 104(4) (1986) 547--571.
	
	\bibitem{smirnov2001critical}
	Stanislav Smirnov.
	\newblock Critical percolation in the plane: conformal invariance, cardy's
	formula, scaling limits.
	\newblock {\em Comptes Rendus de l'Acad{\'e}mie des Sciences-Series
		I-Mathematics}, 333(3) (2001) 239--244.
	
	\bibitem{smirnov2001}
	Stanislav Smirnov and Wendelin Werner.
	\newblock Critical exponents for two-dimensional percolation.
	\newblock {\em Mathematical Research Letters}, 8(6) (2001) 729--744.
	
	
\end{thebibliography}
\end{document}